\documentclass[12pt]{article}
\usepackage{dsfont}
\usepackage{amsfonts}
\usepackage{amssymb}
\usepackage{amsthm}
\usepackage{amsmath}
\usepackage{latexsym}
\usepackage{color}
\usepackage{CJK}

\begin{document}
\newtheorem{lemma}{Lemma}\newcommand{\rco}{\color{red}}
\newtheorem{pron}{Proposition}
\newtheorem{thm}{Theorem}
\newtheorem{cor}{Corollary}
\newtheorem{exam}{Example}
\newtheorem{defin}{Definition}
\newtheorem{rem}{Remark}
\newtheorem{ass}{Assumption}
\newtheorem{thme}{Theorem}
\newcommand{\la}{\frac{1}{\lambda}}
\newcommand{\sectemul}{\arabic{section}}
\renewcommand{\theequation}{\sectemul.\arabic{equation}}
\renewcommand{\thepron}{\sectemul.\arabic{pron}}
\renewcommand{\thelemma}{\sectemul.\arabic{lemma}}
\renewcommand{\thethm}{\sectemul.\arabic{thm}}
\renewcommand{\thethme}{\sectemul.\Alph{thme}}
\renewcommand{\therem}{\sectemul.\arabic{rem}}
\renewcommand{\thecor}{\sectemul.\arabic{cor}}
\renewcommand{\theexam}{\sectemul.\arabic{exam}}
\renewcommand{\thedefin}{\sectemul.\arabic{defin}}
\renewcommand{\theass}{\Alph{ass}}

\numberwithin{equation}{section}\numberwithin{rem}{section}
\numberwithin{cor}{section}
\def\REF#1{\par\hangindent\parindent\indent\llap{#1\enspace}\ignorespaces}
\def\lo{\left}
\def\ro{\right}
\def\be{\begin{equation}}
\def\ee{\end{equation}}
\def\beq{\begin{eqnarray*}}
\def\eeq{\end{eqnarray*}}
\def\bea{\begin{eqnarray}}
\def\eea{\end{eqnarray}}
\def\d{\Delta_T}
\def\r{random walk}
\def\o{\overline}
\def\P{\mathbb{P}}
\def\J{\mathcal{J}}
\title{\large\bf The product of dependent random variables with applications to a discrete-time risk model\thanks{Research supported by National Natural Science Foundation of China
(No. 11401415), and the Priority Academic Program Development of Jiangsu Higher Education
Institutions.} %with dependent financial and insurance risks
}

\author{\small Jikun Chen, Hui Xu, Fengyang Cheng\thanks{ Corresponding author. E-mail:
chengfy@suda.edu.cn}~\\
{\small\it Department of Mathematics, Soochow University, Suzhou, 215006, China}\\
 }
\date{}

\maketitle {\noindent\small {\bf Abstract }}
{\small Let $X$ be a real valued random variable with an unbounded distribution $F$ and let $Y$ be a nonnegative valued random variable with a distribution $G$.  Suppose that $X$ and $Y$ satisfy that
 \begin{eqnarray*}
P(X>x|Y=y)\sim h(y)P(X>x)
\end{eqnarray*}
holds uniformly for $y\geq0$ as $x\to \infty$, where $h(\cdot)$ is a positive measurable function. Under the condition that $\overline{G}(bx)=o(\overline H(x))$ holds for all constant $b>0$, we proved that $F\in\mathcal{L}(\gamma)$ for some $\gamma\geq 0$ implied $H\in \mathcal{L}(\gamma/\beta_G)$ and that $F\in\mathcal{S}(\gamma)$ for some $\gamma\geq 0$ implied $H\in \mathcal{S}(\gamma/\beta_G)$, where $H$ is the distribution of the product $XY$, and $\beta_G$ is the right endpoint of $G$, that is,
   $\beta_G=\sup\{y:~G(y)<1\}\in (0,\infty],$
 and when $\beta_G=\infty$, $\gamma/\beta_G$ is understood as 0.

Furthermore, in a discrete-time risk model in which the net insurance
loss and the stochastic discount factor are equipped with a dependence structure, a general asymptotic formula for the finite-time ruin probability is obtained when the net insurance
loss has a subexponential tail.\\



\section{\bf Introduction  }
\mbox{}

Throughout this paper, all limit relationships are for $x\to\infty$ unless otherwise stated. For two positive functions
$a(\cdot)$ and $b(\cdot)$, we write $a(x)\sim b(x)$ if $\lim a(x)/b(x)=1$, write $a(x)=o(b(x))$ if $\lim a(x)/b(x)=0$,
and write $a(x)=O(b(x))$ if $\limsup a(x)/b(x)<\infty$.
For any  distribution $V$, its tail is denoted by $\overline V(x)=1-V(x)=V(x,\infty)$ for all $x$.
For any r.v. $Z$, its possible value set is denoted by $\mathfrak{D}_Z=\{x\in R:P(Z\in(x-\delta,x+\delta))>0\text{\ for\ all\ }\delta>0\}$.

Throughout this paper, let $X$ be a real valued random variable (r.v.) with a distribution $F$ and let $Y$ be a nonnegative valued r.v. with a distribution $G$, and let $H$ be the distribution of their product $XY$. To avoid triviality, assume that neither $X$ nor $Y$ is degenerate at zero.

\subsection{Some distribution classes}\

A distribution $V$ is said to be unbounded (above) if $\overline{V}(x)>0$ holds for all $x>0$.
An unbounded distribution $V$ supported on $(-\infty,\infty)$ is said to belong to the distribution class $\mathcal{L}(\gamma)$ for some $\gamma\ge0$, if  
\begin{eqnarray}\label{s101}
\overline V(x-t)\sim e^{\gamma t}\overline V(x)
\end{eqnarray}
holds for all constant $t>0$. 
We remark that, when $F$ is lattice and $\gamma>0$, $x$ and $t$ in (\ref{s101}) should be restricted to values of the lattice span.

An unbounded distribution $V$ supported on $(-\infty,\infty)$ is said to belong to the distribution class $\mathcal{S}(\gamma)$ for some $\gamma\ge0$, if  $F\in\mathcal{L}(\gamma)$ and the relation
\begin{eqnarray}\label{s102}
  \overline {V^{*2}}(x)\sim 2c\overline{V}(x)
\end{eqnarray}
holds for some positive number $c$, where $V^{*2}(x)$ denotes the two-fold convolution of $V$ with itself and
$c=\int_{-\infty}^\infty e^{\gamma y}V(dy).$
 When $\gamma=0$, relation (\ref{s101}) represents the well known long-tailed distribution class $\mathcal{L}=\mathcal{L}(0)$, and relations (\ref{s101}) and (\ref{s102}) represent the subexponential distribution class $\mathcal{S}=\mathcal{S}(0)$.

It is well known that if $V\in\mathcal{L}$, then $V$ is heavy-tailed (that is, $\int_0^\infty e^{\alpha y}V(dy)=\infty$ holds for all constant $\alpha>0$) and if $V\in\mathcal{L}(\gamma)$ for some $\gamma>0$, then it is light-tailed (that is, $\int_0^\infty e^{\alpha y}V(dy)<\infty$ for some constant $\alpha>0$).

For reviews on classes $\mathcal{L}(\gamma)$ and $\mathcal{S}(\gamma)$, the readers are referred to Cline and Samorodnitsky (1994), Foss et al. (2013)
and Tang (2006a, 2006b, 2008), among others.

\subsection{Some dependent structures}\

In this subsection, we will introduce some dependence structures.  The first is called Asimit dependent structure, which was introduced by Asimit and Jones (2008).
%The first goal of this paper is to seek relaxed conditions for that the distribution $H$ belongs to $\mathcal{S}(\gamma)$ or $\mathcal{L}(\gamma)$ when $F$ belongs to $\mathcal{S}(\gamma)$ or $\mathcal{L}(\gamma)$, respectively, where $X$ and $Y$ satisfy the following dependence structure.%, which was introduced by Asimit and Jones (2008).
\begin{ass}\label{ass1}
There is a measurable function $h(\cdot): [0, \infty) \rightarrow(0, \infty)$ such that
\begin{eqnarray}
P(X>x|Y=y)\sim h(y)P(X>x) \label{eq201}
\end{eqnarray}
holds uniformly for $y\geq0$ as $x\to \infty$. Here the uniformity is understood as
\begin{eqnarray*}
\lim_{x\rightarrow\infty}\sup_{y\geq0}\Big |\frac{P(X>x|Y=y)}{h(y)P(X>x)}-1\Big |=0.
\end{eqnarray*}
\end{ass}

We remark that, when $y\not\in \mathfrak{D}_Y$%is not a possible value of $Y$, that is, $P(Y\in \Delta)=0$ for some open interval $\Delta$ containing $y$,
, the conditional probability in (\ref{eq201}) is simply understood as unconditional probability, and hence, $h(y) = 1$ for such $y$. Clearly, by (\ref{eq201}), we have $E(h(Y)) = 1$.

 It is worth noting that Assumption \ref{ass1} is a very relaxed dependence structure, which contains many common dependence structure. For example, it contains the following Sarmonov dependent strucure:% such as Farlie-Gumbel-Morgenstern distribution and Sarmonov distribution, etc. Further examples can be found in Examples 1-6 in Asimit and Jones (2008) and Examples 3.1-3.5 in Yang et al. (2016) etc.

We say that a random vector $(X,Y)$ follows a bivariate Sarmonov distribution, that is,
\begin{eqnarray}
P(X\in dx,Y\in dy)=(1+\theta\phi_1(x)\phi_2(y))F(dx)G(dy),\ -\infty<x<\infty,\ y\geq0,\label{sam}
\end{eqnarray}
where $\phi_1(x)$ and $\phi_2(y)$ are two measurable functions %in $(-\infty, \infty)$ % and $\phi_2(x)$ is a continuous function in $\mathfrak{D}_Y$
 and the parameter $\theta$ is a real constant, which satisfy
\begin{eqnarray}\label{sam1}
&&E\phi_1(X)=E\phi_2(Y)=0,
\end{eqnarray}
and
\begin{eqnarray}\label{sam2}
&&1+\theta\phi_1(x)\phi_2(y)\geq0.
\end{eqnarray}

 For more details on multivariate Sarmanov distributions, one can refer to Lee (1996) and
Kotz et al. (2000) among others. There are many choices for the kernels $\phi_1(x)$ and $\phi_2(y)$. One choice is that we take
 $\phi_1(x)=(e^{-x}-a)\mathbb{I}(x>0)$ with $a=\frac{Ee^{-X}\mathbb{I}(X\geq 0)}{P(X\geq 0)}$ and $\phi_2(y)=e^{-y}-E(e^{-Y})$ for all $x\in \mathfrak{D}_X$ and $y\in\mathfrak{D}_Y$, where $\mathbb{I}(A)$ is the indicator function of the set $A$. Another choice is that we take
 $\phi_1(x)=1-2F(x)$ and $\phi_2=1-2G(y)$ for all $x\in \mathfrak{D}_X$ and $y\in\mathfrak{D}_Y$, which leads to another common dependent stucture -- FGM dependent structure:

We say that a random vector $(X,Y)$ follows a Farlie-Gumbel-Morgenstern (FGM) distribution, if
\begin{eqnarray}
P(X>x,Y>y)=\overline F(x)\overline G(y)(1+\theta F(x) G(y))\label{fgm}
\end{eqnarray}
holds for all $x\in \mathfrak{D}_X$ and $y\in\mathfrak{D}_Y$, where $\theta\in[-1,1]$ is a constant.

\subsection{Brief reviews on products}\

 % Since the subexponentiality of the product $XY$ is very important in many applied probability fields such as financial and insurance, risk theory, queueing theory and so on, and has attracted a fair amount of interest in recent years.
Many papers in the literature have been devoted to the tail behavior of the product $Z=XY$. When $X$ and $Y$ are independent, Cline and Samorodnitsky (1994) obtained the following result:
\begin{thme}\label{1a}
Suppose that $X$ and $Y$ are independent r.v.s. If $F\in\mathcal{S}$ and there is a function $a(\cdot):[0,\infty)\to(0,\infty)$ such that

$(a)$ $a(x)\uparrow\infty$;

$(b)$ $a(x)/x\downarrow 0$;

$(c)$ $\overline F(x-a(x))\sim \overline F(x)$;

$(d)$ $\overline G(a(x))=o(\overline H(x))$.\\
Then $H\in\mathcal{S}$.
\end{thme}

 As pointed out by Tang (2006), condition (c) requests that $a(\cdot)$ should not be too small while
condition (d) requests that it should not be too big. Hence, the requirement that conditions (c) and (d) are simultaneously satisfied
is too restrictive in applications. In recent years, many scholars have tried to weaken this restriction, and obtained some interesting results. Under an extra condition that $\limsup \frac{\overline F(vx)}{\overline F(x)}<1,$ Theorem 2.1 in Tang (2006) relaxed conditions (c) and (d) to that
\begin{ass}\label{ass2}$\overline{G}(bx)=o(\overline H(x))$ holds for all constant $b>0$.\end{ass}
If $(X,Y)$ satisfies Assumption \ref{ass2} and $G$ is unbounded, Theorem 2.1 in Tang (2008) proved that $F\in \mathcal{S}(\gamma)$ for some $\gamma>0$ implied that $H\in \mathcal{S}$. Furthermore, Tang (2008) and Xu et al. (2016) obtain the equivalent conditions for $H\in \mathcal{L}$ and $H\in \mathcal{S}$, respectively.% and $F\in \mathcal{L}(\gamma)$ for some $\gamma\geq0$ implied that $H\in \mathcal{L}$ from Theorem 2.1 in Tang (2006)

%$(e)~\overline{G}(bx)=o(\overline H(x))$ holds for all constant $b>0$.
% see Tang (2006a, 2006b, 2008) and Xu et al. (2016) etc.

Meanwhile, since the independence assumption is often too restrictive for practical purposes, many results on the product $XY$ when $X$ and $Y$ follow some dependence structures are obtained. For example, Chen (2011) extended Theorem \ref{1a} to the case that $(X, Y)$ follows an FGM distribution, Under an extra condition, Yang and Sun (2013) extended Theorem \ref{1a} to the case that $(X, Y)$ satisfies Assumption \ref{ass1}, some other results can be found in Yang et al. (2012), Yang and Wang (2013),etc.

Inspired by the above-mentioned research literature, under the relaxed dependence structure Assumption \ref{ass1} and the relaxed condition Assumption \ref{ass2}, we discuss the distribution class of $H$ if  $F\in \mathcal{L}(\gamma)$ and $H\in \mathcal{S}(\gamma)$ for some $\gamma>0$, respectively.

\subsection{Brief reviews on finite-time ruin probabilities in a discrete-time risk model}\label{risk}\

%The second goal of this paper is to give an asymptotic result for finite-time ruin probabilities in a discrete-time risk model with dependent financial and
%insurance risks, details are arranged in Section 3.

%In this section, we will investigate the asymptotic finite-time ruin probabilities for
Consider the following discrete-time risk model with dependent financial and
insurance risks: For any $i\geq 1$, the net insurance loss within period $i$,
which is equal to the total claim amount minus the total premium income, is denoted by a real-valued
random variable $X_i$ and the stochastic discount factor from time $i$ to time $i - 1$ is denoted by a nonnegative r.v. $Y_i$. Hence, the stochastic present values of aggregate net losses up to time $n$ of the insurer can be expressed as
 \begin{eqnarray}
S_0=0,~ \ ~S_n=\sum_{i=1}^nX_i\prod_{j=1}^iY_j, ~   n\in \mathds{N}~\label{1}
\end{eqnarray}
 and the finite-time ruin probability by time $n$ is defined by
\begin{eqnarray}
\psi(x;n) =P\Big(\max_{1\leq m\leq n}\sum_{i=1}^m X_i\prod_{j=1}^i Y_j>x\Big),\label{2}
\end{eqnarray}
where $x> 0 $ can be interpreted as the initial capital.

Recently, a vast amount of papers has been published on this model.  In general, we assume that $(X_i,Y_i),\ i\in \mathds{N}$  is a sequence of independent and identically distributed random vectors with a generic random vector $(X,Y)$.
 Tang and Tsitsiashvili (2003, 2004) discussed the case that $X$ and $Y$ are independent and satisfy conditions of Theorem \ref{1a}, Jiang and Tang (2011),  Chen (2011) considered the case that
  ${(X, Y )}$ follows a bivariate FGM distribution, Tang et al. (2013), Yang and Wang (2013) and Yang et al. (2016) etc. considered the case $X$ and $Y$ satisfy Assumption \ref{ass1} and $F\in \mathcal{L}\cap\mathcal{D}$.
  
  Inspired by the above-mentioned research literature, under the relaxed Assumptions \ref{ass1} and \ref{ass2}, we obtain the asymptotic ruin probability of the above risk model in which $F\in\mathcal{S}$.

The rest of this paper consists of three sections. Section 2 presents the main results. Proofs of theorems and corollaries are arranged in Sections 3 and 4.

\section{\bf Main results}
\

%In this section, we will give some results on the distribution class of the product of dependent r.v.s which satisfy Assumption \ref{ass1}.
% For the sake of convenience, we introduce an assumption as follows:

%\begin{ass}\label{ass2}$\overline{G}(bx)=o(\overline H(x))$ holds for all constant $b>0$.\end{ass}
%%%%%%%%%%%%Theorem 2.1
In this section, we will present the main results of this paper, their proofs are arranged in sections 3 and 4. The first result is to discuss the product of dependent r.v.s:
\begin{thm}\label{th1}Suppose that $(X,Y)$ satisfies Assumptions \ref{ass1} and \ref{ass2}.\\ %that (\ref{th101})  holds for all constant $b>0$.\\
 \textup{(i)}  If $F\in\mathcal{L}(\gamma)$ for some $\gamma\geq 0$,
 then we have that $H\in \mathcal{L}(\gamma/\beta_G)$;\\
 \textup{(ii)}  If $F\in\mathcal{S}(\gamma)$ for some $\gamma\geq 0$,
 then we have that $H\in \mathcal{S}(\gamma/\beta_G)$,\\
 where $\beta_G$ is the right endpoint of $G$, that is,
   $$\beta_G=\sup\{y:~G(y)<1\}\in (0,\infty],$$
 and when $\beta_G=\infty$, $\gamma/\beta_G$ is understood as 0.
\end{thm}
\begin{rem}Clearly, when $\beta_G<\infty$, Assumption \ref{ass2} is automatically satisfied.\end{rem}
From Theorem \ref{th1}, we immediately obtain the following two corollaries:% which has been studied by many scholars:
%%%%%%%%%%%%%%%%%%Corolary 2.1
\begin{cor}\label{co21} Suppose that $(X,Y)$ satisfies Assumption \ref{ass2} and follows a FGM distribution (\ref{fgm}). %Farlie-Gumbel-Morgenstern (FGM) distribution, i.e.
%\begin{eqnarray}
%P(X>x,Y>y)=\overline F(x)\overline G(y)(1+\theta F(x) G(y)),\label{fgm}
%\end{eqnarray}
%where $\theta\in(-1,1)$ is a constant.\\
Suppose that $|\theta|<1$.\\
  \textup{(i)}  If $F\in\mathcal{L}(\gamma)$ for some $\gamma\geq 0$,
 then we have that $H\in \mathcal{L}(\gamma/\beta_G)$;\\
 \textup{(ii)}  If $F\in\mathcal{S}(\gamma)$ for some $\gamma\geq 0$,
 then we have that $H\in \mathcal{S}(\gamma/\beta_G)$.
 \end{cor}

\begin{cor}\label{co22} Assume that $(X,Y)$ satisfies Assumption \ref{ass2} and follows a Sarmonov distribution (\ref{sam})-(\ref{sam2}). %, that is,
%\begin{eqnarray}
%P(X\in dx,Y\in dy)=(1+\theta\phi_1(x)\phi_2(y))F(dx)G(dy),\ -\infty<x<\infty,\ y\geq0,\label{sam}
%\end{eqnarray}
%where $\phi_1(x)$ and $\phi_2(y)$ are two measurable functions %in $(-\infty, \infty)$ % and $\phi_2(x)$ is a continuous function in $\mathfrak{D}_Y$
% and the parameter $\theta$ is a real constant, which satisfy
%\begin{eqnarray}\label{sam1}
%&&E\phi_1(X)=E\phi_2(Y)=0,
%\end{eqnarray}
%and
%\begin{eqnarray}\label{sam2}
%&&1+\theta\phi_1(x)\phi_2(y)\geq0.
%\end{eqnarray}
Suppose that there is a measurable function $\psi(\cdot)$ such that
\begin{eqnarray}
\lim_{\delta\to 0+}\frac{\int_{y-\delta}^{y+\delta}\phi_2(v)G(dv)}{G(y+\delta)-G(y-\delta)}=\psi(y)\label{corb20}
\end{eqnarray}
holds for all $y\in \mathfrak{D}_Y$.
Furthermore, assume that there are constants $c$ and $d_1$ such that
\begin{eqnarray}\label{cor20}\lim_{x\rightarrow\infty}\phi_1(x)=d_1\in(0,\infty)\end{eqnarray}
and
\begin{equation}\label{cor21}P(1+\theta d_1\psi(Y)\geq c)=1.\end{equation}
\textup{(i)}  If $F\in\mathcal{L}(\gamma)$ for some $\gamma\geq 0$,
 then we have that $H\in \mathcal{L}(\gamma/\beta_G)$;\\
 \textup{(ii)}  If $F\in\mathcal{S}(\gamma)$ for some $\gamma\geq 0$,
 then we have that $H\in \mathcal{S}(\gamma/\beta_G)$.
\end{cor}
\begin{rem} If $y$ is a isolated point of set $\mathfrak{D}_Y$ or $y$ is a continuous point in $\mathfrak{D}_Y$ of function $\phi_2(\cdot)$, it is obvious that $\psi(y)=\phi_2(y).$\end{rem}

Now we state the second main results in which the conditions are weaker than that in Tang et al. (2013), Yang and Wang (2013) and Yang et al. (2016) etc.

%%%%%%%%%%%%%%%%%%%Theorem 2.2
\begin{thm}\label{thm2} In the risk model introduced in Subsection \ref{risk}, let $(X_i,Y_i),\ i=1,2,\cdots$  be a sequence of independent and identically distributed random vectors with a generic random vector $(X,Y)$ which satisfies Assumptions \ref{ass1} and \ref{ass2}.
 If $F\in S$, then we have that
\begin{eqnarray}
\Psi(x,n)=P(\max_{1\leq m\leq n}\sum_{i=1}^mX_i\prod_{j=1}^iY_j>x)\sim\sum_{i=1}^n\overline H_i(x)\label{11}
\end{eqnarray}
holds for all $n=1,2,\cdots$, where $H_i(x)=P(X_i\prod_{j=1}^iY_j\leq x)$ for $2\leq i\leq n$ and $H_1=H$.\\
\end{thm}

\begin{cor}\label{co41} In the risk model introduced in Subsection \ref{risk}, let $(X_i,Y_i),\ i=1,2,\cdots$  be a sequence of independent and identically distributed random vectors with a generic random vector $(X,Y)$. Suppose that $(X,Y)$ satisfies Assumption \ref{ass2} and follows an FGM distribution (\ref{fgm}) in which $|\theta|<1$. If $F\in S$, then we have that (\ref{11}) holds for all $n=1,2,\cdots$.
\end{cor}
%\begin{proof}[proof of Corollary \ref{co41}] From the proof of Corollary \ref{co31}, $(X,Y)$ satisfies Assumption \ref{ass1} for $h(\cdot)$ as in (\ref{fgmh}). Hence, Corollary \ref{co41} follows from Theorem \ref{thm2} immediately.  \end{proof}

%%%%%%%%%%%%%%%%%%%%Corollary 3.2

\begin{cor}\label{co42}In the risk model introduced in Subsection \ref{risk}, let $(X_i,Y_i),\ i=1,2,\cdots$  be a sequence of independent and identically distributed random vectors with a generic random vector $(X,Y)$. Assume that $(X,Y)$ satisfies Assumption \ref{ass2} and follows a Sarmonov distribution (\ref{sam})-(\ref{sam2}).
 Furthermore, assume that
 there exist positive constants $c$ and $d_1$ and a function $\psi(\cdot)$ such that (\ref{corb20})-(\ref{cor21}) hold.
 If $F\in S$, then we have that (\ref{11}) holds for all $n=1,2,\cdots$.
\end{cor}
%\begin{proof}  From the proof of Corollary \ref{co32}, $(X,Y)$ satisfies Assumption \ref{ass1} for $h(\cdot)$ as in (\ref{samh}). Hence,
% Corollary \ref{co42} follows from Theorem \ref{thm2}.
%\end{proof}

\section{Proofs of Theorem \ref{th1} and its corollaries}
 Before giving the proof of Theorem \ref{th1}, we first provide two lemmas and their corollaries, which will play a very important role in the proof of Theorem \ref{th1} and their corollaries, and has its own significant value.
\begin{lemma}\label{pr1} Suppose that $(X,Y)$ satisfies Assumptions \ref{ass1} and \ref{ass2}.  Then we have
\begin{eqnarray}
P(XY>x)\sim\int_0^\infty h(y)\overline F\Big(\frac{x}{y}\Big)G(dy)\label{pr11}.
\end{eqnarray}
\end{lemma}
\begin{proof}
According to Lemma 3.2 of Tang (2006), there is a function $b(\cdot):[0, \infty) \rightarrow(0, \infty)$ such that
 $ b(x)\uparrow\infty,
 \frac{b(x)}{x}\downarrow0$ and $\overline G(b(x))=o(\overline H(x))$.
Hence, we have
\begin{eqnarray}
\overline H(x)&=&\int_0^{b(x)}P(X>\frac{x}{y}|Y=y)G(dy)+O(\overline G(b(x)))\nonumber\\
&=&(1+o(1))\int_0^{b(x)}h(y)\overline F\Big(\frac{x}{y}\Big)G(dy)+o(\overline H(x))\nonumber\\
&=&(1+o(1))\Big(\int_0^\infty-\int_{b(x)}^{\infty}\Big)h(y)\overline F\Big(\frac{x}{y}\Big)G(dy)+o(\overline H(x))\label{pr13}.
\end{eqnarray}
From Assumption \ref{ass1}, there exists a constant $x_0>0$ such that $$P(X>x|Y=y)>\frac{1}{2}h(y)P(X>x)$$ holds for all $x>x_0$ and $y\in \mathfrak{D}_Y$, which implies that $h(y)$ is bounded above in $y\in \mathfrak{D}_Y$, that is, there is a positive constant $M$ such that $h(y)\leq M$ holds for all $y\in \mathfrak{D}_Y$. Hence, it follows that
\begin{eqnarray*}
\int_{b(x)}^\infty h(y)\overline F\Big(\frac{x}{y}\Big)G(dy)=O(G(b(x)))=o(\overline H(x))\label{pr14}.
\end{eqnarray*}
Combining with (\ref{pr13}) yields (\ref{pr11}). % follows from (\ref{pr13}) and (\ref{pr14}).
 This ends the proof of Lemma  \ref{pr1}.
\end{proof}
\begin{rem} %A similar result can be found in Lemma 4.1 of Yang and Sun (2013), in which they assumed that $h$ is bounded. This condition is removed in Lemma  \ref{pr1}.
The analysis in relation (\ref{pr11}) shows that the dependence structure of $X$ and $Y$ can be dissolved and
its impact on the tail behavior of quantities under consideration can be captured. More details can be found in the following two corollaries.
\end{rem}

%%%%%%%%%%%%%%%%%%%%%%Corollary 3.1

\begin{cor}\label{co31} Suppose that $(X,Y)$ satisfies Assumption \ref{ass2} and follows an FGM distribution (\ref{fgm}). If  $|\theta|<1$, then we have
\begin{eqnarray}
\overline H(x)\sim\int_0^\infty(1+\theta(G(y)+G(y-)-1))\overline F\Big(\frac{x}{y}\Big)G(dy),\label{cor313}
\end{eqnarray}
where $G(y-)=\lim_{\delta\to 0+}G(y-\delta).$
\end{cor}
\begin{proof}  For any $y\in \mathfrak{D}_Y$, simple calculations yield that
\begin{eqnarray*}
P(X>x|Y=y)&=&\lim_{\delta\to 0+}P(X>x|Y\in (y-\delta,y+\delta])\nonumber\\
&=&  \overline{F}(x)(1+\theta F(x)(G(y)+G(y-)-1)). \label{fgm1}
\end{eqnarray*}
It follows that
\begin{eqnarray}
\frac{P(X>x|Y=y)}{\overline{F}(x)(1+\theta (G(y)+G(y-)-1))}-1=\frac{- \theta \overline{F}(x)(G(y)+G(y-)-1)}{ 1+\theta  (G(y)+G(y-)-1)} . \label{fgm5}
\end{eqnarray}
Clearly, when $|\theta|<1$, we have that
\begin{eqnarray}
P(X>x|Y=y)\sim \overline{F}(x)(1+\theta (G(y)+G(y-)-1)) \label{7}
\end{eqnarray}
holds uniformly in $y\in \mathfrak{D}_Y$.
Hence, relation (\ref{cor313}) follows from Lemma  \ref{pr1} immediately by letting
\begin{equation}\label{fgmh}h(y)=1+\theta (G(y)+G(y-)-1),\ ~y\in \mathfrak{D}_Y.\end{equation}
\end{proof}
\begin{rem}Let $\alpha_G$ denote the left endpoint of $G$, i.e. $\alpha_G=\inf\{y:G(y)>0\}$.

\textup{(i)} If $\theta=1$ and $P(Y=\alpha_G)>0$, from (\ref{fgm5}), we have that (\ref{7}) holds uniformly in $y\in \mathfrak{D}_Y$ also. Hence, (\ref{cor313}) also holds in this case;

\textup{(ii)} If $\theta=-1$ and $P(Y=\beta_G)>0$, from (\ref{fgm5}), we have that (\ref{7}) holds uniformly in $y\in \mathfrak{D}_Y$ also. Hence, (\ref{cor313}) also holds in this case.

\end{rem}
%%%%%%%%%%%%%%%%%%%%Corollary 3.2

\begin{exam}
Suppose that $(X,Y)$ follows an FGM distribution (\ref{fgm}) in which $F$ is unbounded and $Y$ satisfies
$$P(Y=1)=P(Y=2)=\frac{1}{2}.$$
Then, by Corollary \ref{co31}, we have that
$$P(XY>x)\sim \frac{1}{2}\Big (\overline{F}(x)+\overline{F}(x/2)\Big )+ \frac{\theta}{4}\Big (\overline{F}(x/2)-\overline{F}(x)\Big )$$
holds since $$G(y-)+G(y)-1=\left \{ \begin{array}{cc} -\frac{1}{2}, &y=1\\\frac{1}{2}, &
y=2.\end{array}\right .$$ Direct calculations yield that
$$P(XY>x)= \frac{1}{2}\Big (\overline{F}(x)+\overline{F}(x/2)\Big )+ \frac{\theta}{4}\Big (\overline{F}(x/2)-\overline{F}(x)\Big )-\frac{\theta}{4}\Big (\overline{F}^2(x/2)-\overline{F}^2(x)\Big ).$$
\end{exam}

\begin{cor}\label{co32} Assume that $(X,Y)$ follows a Sarmonov distribution (\ref{sam})-(\ref{sam2}).
 Furthermore, assume that
 there exist positive constants $c$ and $d_1$ and a function $\psi(\cdot)$ such that (\ref{corb20})-(\ref{cor21}) hold.
 Then we have
\begin{eqnarray}
P(XY>x)\sim\int_0^\infty(1+\theta d_1\psi(y))\overline F\Big(\frac{x}{y}\Big)G(dy)\label{cor30}.
\end{eqnarray}
\end{cor}
\begin{proof}
For any $y\in \mathfrak{D}_Y$, by (\ref{corb20}), simple calculations yield that
\begin{eqnarray}
P(X>x|Y=y)&=&\lim_{\delta\to 0+}P(X>x|Y\in (y-\delta,y+\delta])\nonumber\\
&=&  \overline{F}(x)+\theta \psi(y)\int_x^\infty \phi_1(u)F(du). \nonumber\label{fgm1}
\end{eqnarray}
It follows from (\ref{cor20}) and (\ref{cor21}) that
\begin{eqnarray*}
P(X>x|Y=y)
\sim(1+\theta d_1\psi(y))\overline F(x)
\end{eqnarray*}
holds uniformly for $y\in \mathfrak{D}_Y$. Hence, relation (\ref{cor30}) follows from Lemma  \ref{pr1} immediately by letting
\begin{equation}\label{samh}h(y)=1+\theta d_1\psi(y),\ ~y\in \mathfrak{D}_Y.\end{equation}
\end{proof}
\begin{rem}\label{rem2}
  A similar relation can be found in Lemma 3.1 of Yang and Wang (2013) in which $F$ has a dominatedly varying tail, i.e.
  $$\limsup \frac{\overline{F}(xy)}{\overline{F}(x)}<\infty$$
  holds for all $0<y<1$. We remark that taking $$\phi_1(x)= 2F(x)-1,~\phi_2(y)= 2G(y)-1$$
  leads to an FGM distribution. In this case the $\psi(\cdot)$ in (\ref{corb20}) is defined as
  \begin{eqnarray*}
\psi(y)=\lim_{\delta\to 0+}\frac{\int_{y-\delta}^{y+\delta}\phi_2(v)G(dv)}{G(y+\delta)-G(y-\delta)}= G(y)+G(y-)-1
\end{eqnarray*}
for $y\in \mathfrak{D}_Y$ and $d_1=1$ in (\ref{cor20}). Furthermore, if $|\theta|<1$, then we can take $c=1-|\theta|>0$ in (\ref{cor21}).

\end{rem}

\begin{lemma}\label{lem0}
Suppose that $X$ and $Y$ are independent r.v.s. satisfying Assumption \ref{ass2}.\\ %Assume that $F\in\mathcal{S}$ and $G$ is unbounded, then
\textup{(i)}  If $F\in\mathcal{L}(\gamma)$ for some $\gamma\geq 0$,
 then we have that $H\in \mathcal{L}(\gamma/\beta_G)$;\\
 \textup{(ii)}  If $F\in\mathcal{S}(\gamma)$ for some $\gamma\geq 0$,
 then we have that $H\in \mathcal{S}(\gamma/\beta_G)$.
 \end{lemma}

  \begin{proof}Since $Y$ is nonnegative, we have that
  $$P(XY>x)=P(X^+Y>x)$$
  holds for all $x>0$, where  $X^+=XI(X\geq 0)$ and  $I(A)$ is the indicator function of the set $A$.
From Corollary 1.1 of Tang (2008), Assumption \ref{ass2} implies that
 $\overline G(x/d)-\overline G((x+1)/d)=o(\overline H(x))$
holds for all $d\in D[F]$ if $D[F]\neq\emptyset $, where $D[F]$ represents the set of all positive discontinuities of $F$. Hence, part (i) follows from Theorem 1.1 of Tang (2008)(for $\beta_G=\infty$) and from Lemma A.4 of Tang and Tsitsiashvili(2004)(for $\beta_G<\infty$); When $\beta_G=\infty$, part (ii) follows from Theorem 1.1 in Xu et al. (2016)(for $\gamma=0$) and from Theorem 1.2 of Tang (2008)(for $\gamma>0$); and when $\beta_G<\infty$, part (ii) follows from Theorem 1.1 of Tang (2006a).\end{proof}

  Now we are standing in a position to prove Theorem \ref{th1}.
%%%%%%%%%%%%%%%%%%%%%%%Proof of Theorem 2.1
\begin{proof}[Proof of Theorem \ref{th1}.] Let $Y_h$ be an r.v. with a distribution $G_h(dy)=h(y)G(dy)$, independent of $X$. Clearly, $G_h$ is a proper distribution since $E(h(Y)) = 1$. Hence, from Lemma \ref{pr1},  we have that
\begin{eqnarray}
P(XY>x) &\sim& P(XY_h>x)\label{pfth1},
\end{eqnarray}
  Note that $h(y)$ is bounded above in $y\in \mathfrak{D}_Y$, that is,   there is a positive constant $M$ such that $h(y)\leq M$ holds for all $y\in \mathfrak{D}_Y$.
Hence, we have that $\overline G_h(bx)=o(P(XY_h>x))$ holds for all constant $b>0$ since
 $$\overline G_h(bx)=\int_{bx}^\infty h(y)G(dy)\leq M\overline G(b x).$$
 On the other hand, since $h$ is positive on $\mathfrak{D}_Y$, it is obvious that $G$ and $G_h$ have a common right endpoint, i.e. $\beta_{G_h}=\beta_G$.
 Hence, Theorem \ref{th1} follows from Lemma \ref{lem0} and (\ref{pfth1}).
 \end{proof}
 
  At the end of this section, we give the proofs of Corollaries \ref{co21} and \ref{co22}:
\begin{proof}[Proof of Corollary \ref{co21}] From the proof of Corollary \ref{co31}, $(X,Y)$ satisfies Assumption \ref{ass1} for $h(\cdot)$ as in (\ref{fgmh}). Hence, Corollary \ref{co21} follows from Theorem \ref{th1} immediately.
\end{proof}
\begin{proof}[Proof of Corollary \ref{co22}] From the proof of Corollary \ref{co32}, $(X,Y)$ satisfies Assumption \ref{ass1} for $h(\cdot)$ as in (\ref{samh}). Hence, Corollary \ref{co22} follows from Theorem \ref{th1} immediately.
\end{proof}

\section{Proof of Theorem \ref{thm2} and its corollaries}
\

%Now we state the second main results in which the conditions are weaker than that in Tang et al. (2013), Yang and Wang (2013) and Yang et al. (2016) etc.
 
In this section, we will prove Theorem \ref{thm2} and its corollaries.
\begin{proof}[Proof of Theorem \ref{thm2}]
 We use a similar method as in the proof of Theorem 3.1 of Chen (2011). %Let $S_n=\sum_{i=1}^nX_i\prod_{j=1}^iY_j$,
Note that $$P(S_n>x)\leq\Psi(x,n)\leq P(\sum_{i=1}^mX_i^+\prod_{j=1}^iY_j>x),$$ where $X_i^+$ denote the positive part of $X_i,i=1,2,\cdots $. Recall that% the following simple fact
\begin{eqnarray*}
S_n=\sum_{i=1}^nX_i\prod_{j=1}^iY_j\overset{\textup{d}}{=}\sum_{i=1}^nX_i\prod_{j=i}^nY_j:=T_n,\label{16}
\end{eqnarray*}
due to the i.i.d. assumption for the sequence $\{(X_i,Y_i),i\geq 1\}$, where $\overset{\textup{d}}{=}$ stands for equality in distribution.
Therefore, we only need to prove that the relation
\begin{eqnarray}
P(T_n>x)\sim\sum_{i=1}^n\overline H_i(x)\label{12}
\end{eqnarray}
holds. We proceed by induction on $n$: Note that\\
(1:1) $ \overline G(bx)=o(\overline H(x))$ holds for all constant $b>0$,\\%\text{\ and\ } \overline H(x-b(x)\sim \overline H(x))\\
(1:2) relation (\ref{12}) trivially holds for $n=1$, \\
(1:3) $ T_1=X_1Y_1$  follows a subexponential distribution.

 Now we assume that:\\
(n:1) $\overline G(bx)=o(\overline H_n(x))$ holds for all constant $b>0$,\\%\text{\ and\ } \overline H_n(x-b(x))\sim\overline H_n(x)\\
(n:2) relation (\ref{12}) holds for $n$,\\
(n:3) both $ X_n\prod_{j=1}^nY_j$ and $T_n$ follow subexponential distributions.

 We aim to prove that:\\
(n+1:1) $\overline G(bx)=o(\overline H_{n+1}(x))$ holds for all constant $b>0$,\\%\text{\ and\ }\overline H_{n+1}(x-b(x))\sim\overline H_{n+1}(x)\\
(n+1:2) relation (\ref{12}) holds for  n+1,\\
(n+1:3) both $X_{n+1}\prod_{j=1}^{n+1}Y_j$ and $T_{n+1}$ follow subexponential distributions.

Using a similar method as in the proof of Theorem 3.1 of Chen (2011), we have that
\begin{eqnarray}
P(T_n+X_{n+1}>x)\sim P(T_n>x)+P(X_{n+1}>x) \label{15}
\end{eqnarray}
holds since $T_n$ and $X_{n+1}$ are independent.
Consider $(n+1:1)$, Since $ (X_i,Y_i),i=1,2,\cdots$ are i.i.d. random vectors,  we have
\begin{eqnarray}
X_{n+1}\prod_{i=1}^{n+1}Y_i\overset{\textup{d}}{=}\Big(X_n\prod_{i=1}^nY_j\Big)Y_{n+1},\label{th210}
\end{eqnarray}
which yields that $\overline G(bx)=o(\overline H_{n+1}(x))$ holds for all constant $b>0$ from Corollary 1.1 of Tang(2008) and the induction hypothesis (n:1).

Next we consider $(n+1:2)$.
From Lemma 3.2 of Tang (2006) and Assumption \ref{ass2}, there exists a function $b(\cdot):[0,\infty)\to(0,\infty)$ such that
$ b(x)\uparrow\infty,
 \frac{b(x)}{x}\downarrow0$ and
\begin{equation}\label{lem101}\overline G(b(x))=o(P(X_{n+1}Y_{n+1}>x)).\end{equation}
 Note that
\begin{eqnarray}
 P(T_{n+1}>x)&=&P((X_{n+1}+T_n)Y_{n+1}>x)\nonumber\\
 &=& \int_0^{b(x)}P(X_{n+1}+T_n>\frac{x}{y}|Y_{n+1}=y)G(dy)+O(\overline G(b(x))). \label{th201}
\end{eqnarray}
  Since both $X_{n+1}$ and $T_n$ follow long-tailed distributions, by Corollary 2.5 of Cline and Samorodnitsky (1994), there is a function $a(\cdot):[0,\infty)\to(0,\infty)$ such that $a(x)\uparrow\infty,
 \frac{a(x)}{x}\downarrow0$ and
  \begin{equation}\label{pfth201}P(X_{n+1}>x-a(x))\sim P(X_{n+1}>x)~\textup{ and }~P(T_n>x-a(x))\sim P(T_n>x).\end{equation}
 We split the conditional probability $P(X_{n+1}+T_n>\frac{x}{y}|Y_{n+1}=y)$ in (\ref{th201}) into four parts as
\begin{eqnarray}
P(X_{n+1}+T_n>\frac{x}{y}|Y_{n+1}=y)=I_1(x,y)+I_2(x,y)+I_3(x,y)+I_4(x,y)\label{19},
\end{eqnarray}
where
\begin{eqnarray*}
&&I_1(x,y)=P(X_{n+1}+T_n>\frac{x}{y},T_n\in[-a\Big(\frac{x}{y}\Big),a\Big(\frac{x}{y}\Big)]|Y_{n+1}=y),\\
&&I_2(x,y)=P(X_{n+1}+T_n>\frac{x}{y},T_n>\frac{x}{y}-a\Big(\frac{x}{y}\Big)|Y_{n+1}=y),\\
&&I_3(x,y)=P(X_{n+1}+T_n>\frac{x}{y},T_n<-a\Big(\frac{x}{y}\Big)|Y_{n+1}=y),\\
&&I_4(x,y)=P(X_{n+1}+T_n>\frac{x}{y},T_n\in[a\Big(\frac{x}{y}\Big),\frac{x}{y}-a\Big(\frac{x}{y}\Big)]|Y_{n+1}=y).
\end{eqnarray*}
First we estimate $I_1(x,y)$: Since $T_n$ is independent of $(X_{n+1},Y_{n+1})$, from Assumption \ref{ass1} and (\ref{pfth201}), we have that
 \begin{eqnarray}
 I_1(x,y)&=&\int_{-a(\frac{x}{y})}^{a(\frac{x}{y})}P(X_{n+1}>\frac{x}{y}-u|Y_{n+1}=y)P(T_n\in du)\nonumber\\
 &\sim& \int_{-a (\frac{x}{y} )}^{a (\frac{x}{y} )}h(y)P(X_{n+1}>\frac{x}{y}-u )P(T_n\in du)\nonumber\\
 &\sim& \int_{-a (\frac{x}{y} )}^{a (\frac{x}{y} )}h(y)P(X_{n+1}>\frac{x}{y})P(T_n\in du)\nonumber\\
 &=&h(y)P(X_{n+1}>\frac{x}{y})P(T_n\in[-a\Big(\frac{x}{y}\Big),a\Big(\frac{x}{y}\Big)])\nonumber\\
 &\sim& h(y)P(X_{n+1}>\frac{x}{y})\label{20}
  \end{eqnarray}
  holds uniformly for all $y \in(0,b(x))$. Here the uniformity is understood as
  $$\lim_{x\to\infty} \sup_{y \in(0,b(x))}\left |\frac{  I_1(x,y)}{h(y)P(X_{n+1}>\frac{x}{y})}-1\right |=0.$$

Next, we estimate $I_2(x,y)$: Note that
  \begin{eqnarray*}
 P(X_{n+1}+T_n>\frac{x}{y},X_{n+1}\in\Big[ -a\Big(\frac{x}{y}\Big),a\Big(\frac{x}{y}\Big)\Big]|Y_{n+1}=y)\leq I_2(x,y)\leq P(T_n>\frac{x}{y}-a\Big(\frac{x}{y}\Big)|Y_{n+1}=y),
 \end{eqnarray*}
  Since $T_n$ and $(X_{n+1},Y_{n+1})$ are independent, we can easily obtain that
 \begin{eqnarray}
 I_2(x,y)\sim P(T_n>\frac{x}{y})\sim \sum_{i=1}^n\overline{H}_i\Big(\frac{x}{y}\Big)\label{21}
 \end{eqnarray}
 holds uniformly for $y\in(0,b(x))$. In fact,
 \begin{eqnarray*}
&& P(X_{n+1}+T_n>\frac{x}{y},X_{n+1}\in\Big[ -a\Big(\frac{x}{y}\Big),a\Big(\frac{x}{y}\Big)\Big]|Y_{n+1}=y)\\
 &=&\int_{-a(\frac{x}{y})}^{a(\frac{x}{y})}P\Big(T_n>\frac{x}{y}-u\Big)P(X_{n+1}\in du|Y_{n+1}=y)\\
 &\sim&\int_{-a(\frac{x}{y})}^{a(\frac{x}{y})}P\Big(T_n>\frac{x}{y}\Big)P(X_{n+1}\in du|Y_{n+1}=y)\\
 &=&P\Big(T_n>\frac{x}{y}\Big)P\Big(X_{n+1}\in\Big[ -a\Big(\frac{x}{y}\Big),a\Big(\frac{x}{y}\Big)\Big]|Y_{n+1}=y\Big).
 \end{eqnarray*}
Now we estimate $I_3(x,y)$: It follows that
\begin{eqnarray}
I_3(x,y)&\leq& P(X_{n+1}>\frac{x}{y}+a\Big(\frac{x}{y}\Big),T_n\leq -a\Big(\frac{x}{y}\Big) |Y_{n+1}=y)\nonumber\\
&\sim& h(y)P(X_{n+1}>\frac{x}{y})P(T_n\leq -a\Big(\frac{x}{y}\Big))\nonumber\\
&=&o(h(y)P(X_{n+1}>\frac{x}{y}))\label{22}.
\end{eqnarray}
holds uniformly for all $y\in(0,b(x))$. Here the uniformity is understood as
  $$\lim_{x\to\infty} \sup_{y \in(0,b(x))}\frac{  I_3(x,y)}{h(y)P(X_{n+1}>\frac{x}{y})}=0.$$
Finally, we estimate $I_4(x,y)$:
 It follows from (\ref{15}) that
\begin{eqnarray}
I_4(x,y)&=&\int_{a (\frac{x}{y} )}^{\frac{x}{y}-a (\frac{x}{y} )}P(X_{n+1}>\frac{x}{y}-u|Y_{n+1}=y)P(T_n\in du)\nonumber\\
&\sim& \int_{a (\frac{x}{y} )}^{\frac{x}{y}-a (\frac{x}{y} )}h(y)P(X_{n+1}>\frac{x}{y}-u)P(T_n\in du)\nonumber\\
&=&h(y)P(X_{n+1}+T_n>\frac{x}{y},a\Big(\frac{x}{y}\Big)\leq T_n\leq \frac{x}{y}-a\Big(\frac{x}{y}\Big))\nonumber\\
&=&o(h(y)(P(X_{n+1}>\frac{x}{y})+P(T_n>\frac{x}{y})))\label{23},
\end{eqnarray}
holds uniformly for all $y\in (0,b(x))$.
 Substituting relations (\ref{20})-(\ref{23}) into (\ref{19}), we have that
\begin{equation*}P(X_{n+1}+T_n>\frac{x}{y}|Y_{n+1}=y)\sim h(y)P(X_{n+1}>\frac{x}{y})+\sum_{i=1}^n\overline{H}_i\Big(\frac{x}{y}\Big)%P(T_n>\frac{x}{y})
\end{equation*}
holds uniformly for all $y\in(0,b(x))$, which yields that%. Taking it into (\ref{18}) we get
\begin{eqnarray*}
&&\int_0^{b(x)}P(X_{n+1}+T_n>\frac{x}{y}|Y_{n+1}=y)G(dy)\\&\sim& \int_0^{b(x)}P(X_{n+1}>\frac{x}{y}|Y_{n+1}=y)G(dy)+\sum_{i=1}^n\int_0^{b(x)}P(X_i\prod_{j=1}^iY_j>\frac{x}{y})G(dy)\\
&\sim &\sum_{i=1}^{n+1}\overline H_i(x)+O(\overline{G}(b(x))\\
&\sim &\sum_{i=1}^{n+1}\overline H_i(x).
\end{eqnarray*}

Finally, we consider $(n+1:3)$. By (\ref{th210}), $X_{n+1}\prod_{i=1}^{n+1}Y_i\in \mathcal{S}$ follows from $H_n\in\mathcal{S}$ and Lemma \ref{lem0}.
The proof of that $T_{n+1}$ follows a subexponential distribution is similar to that of Theorem 3.1 in Chen (2011), we omitted the detail.
\end{proof}

At the end of this section, we give the proofs of Corollaries \ref{co41} and \ref{co42}:
\begin{proof}[Proof of Corollary \ref{co41}.] From the proof of Corollary \ref{co31}, $(X,Y)$ satisfies Assumption \ref{ass1} for $h(\cdot)$ as in (\ref{fgmh}). Hence, Corollary \ref{co41} follows from Theorem \ref{thm2} immediately.  \end{proof}

\begin{proof}[Proof of Corollary \ref{co42}.]  From the proof of Corollary \ref{co32}, $(X,Y)$ satisfies Assumption \ref{ass1} for $h(\cdot)$ as in (\ref{samh}). Hence,
 Corollary \ref{co42} follows from Theorem \ref{thm2}.
\end{proof}

%{\bf Acknowledgments }\\

\end{document}